\documentclass[12pt]{amsart}

\usepackage[usenames,dvipsnames]{pstricks} 
\usepackage{epsfig}
\usepackage{graphicx,color}
\usepackage{geometry}
\usepackage[all]{xy}
\usepackage{amssymb}
\usepackage{cite}
\usepackage{fullpage}
\usepackage{commath}
\usepackage{comment}
\usepackage{skull}
\xyoption{poly}

\swapnumbers 
\newtheorem{theorem}{Theorem}[section]
\newtheorem{lemma}[theorem]{Lemma}
\newtheorem{proposition}[theorem]{Proposition}
\newtheorem{corollary}[theorem]{Corollary}
\theoremstyle{definition}
\newtheorem{definition}[theorem]{Definition}

\newtheorem{example}[theorem]{Example}

\newtheorem*{theorem*}{Theorem}
\newtheorem*{remark*}{Remark}
\newtheorem*{remarks*}{Remarks}
\newtheorem*{definition*}{Definition}
\newtheorem{question}[theorem]{Question}

\newtheorem{defn}[theorem]{Definition}
\newtheorem{thm}[theorem]{Theorem}
\newtheorem{cor}[theorem]{Corollary}
\newtheorem{prop}[theorem]{Proposition}

\newtheoremstyle{named}{}{}{\itshape}{}{\bfseries}{.}{.5em}{\thmnote{#3}#1}
\theoremstyle{named}
\newtheorem*{namedtheorem}{}

\usepackage{calrsfs}

\usepackage[T1]{fontenc}
\usepackage{textcomp}
\usepackage{times}
\usepackage[scaled=0.92]{helvet}

\newcommand{\OO}{\mathcal{O}}    
\newcommand{\ZZ}{\mathbb{Z}}     
\newcommand{\PP}{\mathbb{P}}      
\newcommand{\QQ}{\mathbb{Q}}      
\newcommand{\CC}{\mathbb{C}}      
\newcommand{\pp}{\mathfrak{p}}   

\newcommand{\PGL}{\mathrm{PGL}}

\DeclareMathOperator{\ord}{ord}

\newcommand{\Hawaii}{Hawai\kern.05em`\kern.05em\relax i}

\begin{document}

\date{\today}
\title[Some Applications of Dynamical Belyi Polynomials]{Some Applications of Dynamical Belyi Polynomials}
\author{Jacqueline Anderson}
\author{Michelle Manes}
\author{Bella Tobin}

\address{Department of Mathematics, 
Bridgewater State University, 
Bridgewater, MA 02325,
ORCiD: 0000-0002-4732-4251}

\address{American Institute of Mathematics, 
1200 E California Blvd,
Caltech 8-32,
Pasadena, CA 91125,
ORCiD: 0000-0001-6843-9157}

\address{Department of Mathematics, Agnes Scott College, 141 E College Ave., Decatur, GA 30030, ORCiD: 0000-0002-3207-8537}
\email{jacqueline.anderson@bridgew.edu}
\email{mmanes@aimath.org}
\email{btobin@agnesscott.edu}


\maketitle

 \begin{abstract}
We give necessary and sufficient conditions for post-critically finite polynomials to have persistent bad reduction at a given prime. We also answer in the negative a pair of questions posed by Silverman about conservative polynomials. Our proofs rely on conservative dynamical Belyi polynomials as exemplars of PCF (resp. conservative) maps. 
\end{abstract}

 New questions and conjectures in the field of arithmetic dynamics are often drawn from a motivating analogy between 
 objects in arithmetic geometry and objects in dynamical systems. For example, torsion points on abelian varieties parallel periodic points for an iterated morphism of projective space $\PP^n$, and rational points on an elliptic curve parallel rational points in orbits of a morphism of $\PP^1$. 
  
 An abelian variety has complex multiplication (CM) if it has a larger-than-expected endomorphism ring; that is, if it has additional symmetries beyond what one would expect. The additional structure of CM abelian varieties along with properties arising from this structure leads them to serve as valuable test objects in arithmetic geometry.  For example, we know that a CM abelian variety $A/\CC$ is in fact algebraic, i.e., it is  defined over some number field $K$. Moreover, there exists a finite extension $L/K$ such that the abelian variety, when base-changed to $L$, has everywhere good reduction. Roughly speaking, arithmetic objects have ``good reduction'' if they are well-behaved when the object is considered modulo a prime, and they have ``everywhere good reduction'' if this holds for every prime.

 A morphism $f: \PP^1 \to \PP^1$ is {\bf post-critically finite} (PCF) if each of its critical points has a finite forward orbit under iteration of $f$. 
In~\cite{MdADS}, Silverman  proposes  PCF functions as dynamical analogues of CM abelian varieties,   but the strength of this analogy is not fully understood.  From Thurston's rigidity theorem~\cite{ThurstonProof}, we know that outside of a single well-understood family, post-critically finite functions on $\PP^1$ defined over $\CC$ are algebraic.
 In this paper, we use a particular class of post-critically finite polynomials to answer three questions about PCF maps to better understand the role that these functions might have as valuable test objects in arithmetic dynamics.
 
The first of these  questions concerns the concept of (potential) good reduction and is motivated by the analogy to CM abelian varieties.  
It is known that a PCF polynomial of degree $d\geq 2$ has potential good reduction at a prime $p$ if $p > d$ or if $d = p^n$ (see, for example, \cite{anderson} and\cite{epstein}). It is natural to ask if this is sharp. That is, given a degree $d$ and prime $p$ such that $p < d$ and $d \neq p^n$, does there exist a PCF polynomial $f(z) \in \overline\QQ[z]$ of degree $d$ such that $f$ does not have potential good reduction at $p$? In Section~\ref{sec:badred}, we provide a complete answer to this question and provide an example of a PCF polynomial with persistent bad reduction in every case in which it is possible.

\begin{namedtheorem}[Theorem 1]\label{thm:testref}\label{thm:badred}
Let $d, p \in \ZZ$, with $d \geq 2$ and $p$ prime. Write $d=p^n\ell$ with $n\geq 0$ and $p\nmid \ell$. Then there exists a post-critically finite polynomial $f \in \overline{\QQ}[z]$ of degree $d$ with persistent bad reduction at $p$ if and only if $\ell >p$.
\end{namedtheorem}

Height functions are ubiquitous in arithmetic geometry as a measure of the arithmetic complexity of objects. If PCF functions are in fact special arithmetically, one might expect their heights to reflect this. Indeed, in~\cite{BIJL} the authors show that PCF maps of degree $d\geq 2$ form a set of bounded height in the moduli space $\mathcal M_d$ of degree-$d$ rational functions, but the height bound depends on the degree of the map.
{\bf Conservative polynomials} are a subclass of PCF polynomials in which every critical point is fixed. In~\cite{MdADS}, Silverman considers monic conservative polynomials, normalized so that $f(0)=0$, and asks if the heights of such polynomials grow at a slow rate relative to their degrees. 

More precisely, let $\mathcal{C}_d^{\text{poly}}$ be the set of of normalized conservative polynomials of degree $d$ in $\overline{\mathbb{Q}}[z]$ and let $h(f)$ be the height of a polynomial (see Section~\ref{sec:heights} for a precise definition). Sliverman asks if the following statements are true:

\begin{align*}
\lim_{d \to \infty} \max_{f \in \mathcal{C}_d^{\text{poly}}} h(f) &= 0?\\
\limsup_{d \to \infty} \max_{f \in \mathcal{C}_d^{\text{poly}}} \frac{h(f)}{(\log{d})/d} &< \infty?
\end{align*}
In Section~\ref{sec:heights}, we provide a negative answer to both of these questions.  



Our results use conservative dynamical Belyi polynomials as a primary tool. These polynomials, studied  in~\cite{manesbelyi}, have exactly two finite critical points, both of which are fixed by the polynomial. The third author used dynamical Belyi maps to describe and study  bicritical polynomials in her Ph.D. thesis~\cite{TobinThesis}. The applications of these functions thus far suggest that they may indeed prove to be a fertile testing ground for arithmetic dynamics.

\section{Background}\label{sec:bg}
\subsection{Dynamical systems definitions and notation}
Let $K$ be a field and let  $f(z) \in K[z]$ have degree $d\geq 2$.   
Critical points of $f$ are the points $\alpha \in \overline K$ such that $f'(\alpha) = 0$. The \emph{orbit} of a point $a \in K$ under $f$ is the set obtained by iterating the function $f$ starting with the point $a$: $\{ a, f(a), f(f(a)), \dots \}$. 
\begin{defn}
A polynomial $f$ is {\bf post-critically finite} (PCF) if the orbit of each critical point is finite.  A polynomial is {\bf post-critically bounded} with respect to a given absolute value if the orbit of each critical point is bounded with respect to that absolute value. 
\end{defn}

 Let $\phi(z) = az+b \in \overline K[z]$ be an affine change of coordinates on $\overline K$ and define $f^\phi := \phi^{-1} \circ f \circ \phi \in \overline K[z]$. We say that $f$ and $f^\phi$ are conjugates of each other. This conjugation provides a natural dynamical equivalence relation on polynomials since it respects iteration: $\left(f^\phi\right)^n = \left(f^n\right)^\phi$ for all $n\geq 1$. 

\subsection{Conservative dynamical Belyi polynomials}

A dynamical Belyi map is a morphism $f : \PP^1 \to \PP^1$ ramified only over $\{0, 1, \infty \}$ such that $f\left(\{0, 1, \infty \}\right) \subseteq \{0, 1, \infty \}$. If $f$ is a polynomial, then $\infty$ is a totally ramified fixed point.  A {\bf conservative} polynomial is one where every critical point is fixed. So a conservative dynamical Belyi polynomial is one where the only possible finite critical points are $0$ and $1$, and where each of those points is a fixed point. Thus, conservative dynamical Belyi polynomials are examples of PCF polynomials.

Let $f \in \CC[z]$ be a polynomial of degree $d$. Let $e_\alpha(f)$ represent the ramification index of $f$ at~$\alpha$; that is, $e_\alpha(f) = \ord_{z=\alpha}(f(z) - f(\alpha))$. If $f$ is a conservative dynamical Belyi polynomial, then the Riemann-Hurwitz Formula for $\PP^1$~\cite[Theorem 1.1]{ADS} tells us that $e_0(f) + e_1(f) = d+1$. From~\cite[Proposition 1]{manesbelyi}, if we fix $d$, $e_0 $, and $e_1 $ subject to this constraint, there is a unique conservative dynamical Belyi polynomial with this ramification data,  
and in fact that polynomial is defined over $\QQ$. We can then take this unique polynomial as our definition of a conservative dynamical Belyi polynomial. In the definition below, for ease of notation we have $e_0 = d-k$ and $e_1 = k+1$.

\begin{definition}[{~\cite[Proposition 2]{manesbelyi}}]\label{def:belyi} For integers $d$ and $k$ such that $d \geq 3$ and $1 \leq k \leq d-2$, define the conservative dynamical Belyi polynomial $B_{d,k}$ as follows:
\begin{equation}\label{eqn:DynBeli}
 B_{d,k}(z) = \sum_{i=0}^k a_iz^{d-k+i}, \text{ where }  a_i = (-1)^{i} {d\choose {k-i}}{{d-k+i-1}\choose i}.
 \end{equation}

\end{definition}
There are $d-2$ conservative dynamical Belyi polynomials for each degree $d \geq 3$, and they all have integer coefficients.
\begin{example}\label{ex:deg4} For example, there are two conservative dynamical Belyi polynomials of degree 4, obtained by choosing $k=1$ or $k=2$ in the formula above:
\[ B_{4,1}(z) = -3z^4+4z^3 \]
\[ B_{4,2}(z) = 3z^4-8z^3+6z^2.\]

Note that these two polynomials are conjugates of each other, as one can be obtained from the other by conjugating by the linear map $\phi(z) = 1-z$. More generally, $B_{d,k}$ is conjugate to $B_{d,d-k-1}$ via the same change of coordinates swapping 0 and 1.
\end{example}

In the sections that follow, we explore properties of these conservative dynamical Belyi maps and use them to answer questions related to post-critically finite maps and conservative polynomials in general. We refer the reader to ~\cite{manesbelyi} and~\cite[Section 3]{PilgramDessins} for further background on dynamical Belyi maps.


\section{Post-Critically Finite Polynomials and Bad Reduction}\label{sec:badred}

 In order to properly talk about good and bad reduction of polynomials, we will work in a non-archimedean field. We can then translate results to a number field $L$ by considering the completion $L \hookrightarrow L_\nu$ where $\nu$ is any absolute value on $L$. Since any polynomial $f\in \overline{\QQ}[z]$ is in fact defined over a number field, this allows us to make sense of good and bad reduction for polynomials with algebraic coefficients. We set the following notation:
 
 \begin{tabular}{ll}
 $K$ & a non-archimedean field, complete with respect to an absolute value $\nu$.\\
$\OO_K$ & the ring of integers $\{ \alpha \in K \colon | \alpha|_\nu \leq 1\}$.\\
$\pp$ & the maximal ideal of $\OO_K$.\\
$k$ & the residue field of $K$; that is $\OO_K / \pp\OO_K$.\\
$\overline f$ & the polynomial in $k[z]$ obtained by reducing the coefficients of $f \in \OO_K[z]$ modulo $\pp$.
\end{tabular}



Let $f \in K(z)$. We may choose $\phi(z)\in \PGL_2(\overline{K})$ so that $f^\phi \in \mathcal O_K(z)$ and at least one coefficient of $f$ is a unit. 
\begin{definition}
We say that $f$ has {\bf good reduction} if $\deg(\overline f) = \deg(f)$, and $f$ has {\bf potential good reduction} if there is some $\phi(z) \in \PGL_2(\overline{K})$, such that $f^\phi$ has good reduction. If $f$ does not have good reduction, then it has {\bf bad reduction}. If $f$ does not have potential good reduction, we say it has {\bf persistent bad reduction}.
\end{definition}
The polynomial $f$ has persistent bad reduction if every polynomial conjugate to $f$ has bad reduction; that is, it has bad reduction in every coordinate.
%
To prove our main theorem, we will use the following lemma of Benedetto. 

\begin{lemma}[\protect{\cite[Corollary 4.6]{MR1813109}}]
\label{lem:pgr} 

Let $f\in K[z]$ be a polynomial, and let $g$ be a polynomial conjugate of $f$ such that $g$ is monic and $g(0)=0$. Then $f$ has potential good reduction if and only if $g$ has  good reduction.

\end{lemma}

We begin by showing that all conservative dynamical Belyi polynomials have persistent bad reduction at some prime. 

\begin{prop}\label{prop:monicconj}
 Let $B_{d,k}(z)$ be a conservative dynamical Belyi polynomial. Then there exists some prime $p$ such that $B_{d,k}$ has persistent bad reduction at $p$. 
\end{prop}

\begin{proof}
Write $B_{d,k}(z) = \sum_{i =0}^k a_i z^{d-k+i} \in \ZZ[z]$ with the $a_i$ given in equation~\eqref{eqn:DynBeli}.  
Note that  $a_k = (-1)^k \binom{d-1}{k}$. Since $1 \leq k \leq d-2$,  we see that  $a_k \neq \pm1$.
The fact that $B_{d,k}(1) = 1$ means that the coefficients $a_i$ do not have a common prime factor, since any such factor would divide the linear combination  $B_{d,k}(1) $.
Therefore we can find a prime $p$ and a value $i \neq k$ such that $p \mid a_k$ and $p\nmid a_i$. 
Choose $\phi \in \PGL_2$ such that $\phi (z) = \frac{ z}{\beta}$ where $\beta^{d-1} = a_k$.  Then 
\begin{equation}\label{eq:BelyConj}
B_{d,k}^\phi =\sum\limits_{i =0}^k \frac{a_i}{\beta^{d-k+i-1}} z^{d-k+i} \in \QQ(\beta)[z].
\end{equation}
Note that $B_{d,k}^\phi(z)$ is monic and fixes $0$. Let $v$ denote any valuation extending the $p$-adic valuation to $\QQ(\beta)$.
Then 
\[
v (a_i) = 0 \text{ and } v( \beta^{d-k+i-1} ) >0, \text{ so } v \left( \frac{a_i}{\beta^{d-k+i-1}} \right) <0.\]
 Hence $B_{d,k}^\phi$ has bad reduction at $p$, so by Lemma~\ref{lem:pgr}, $B_{d,k}$ has persistent bad reduction at $p$. 
\end{proof}

Note: Recall that the {\bf content} of a polynomial with integer coefficients is the greatest common divisor of the coefficients. 
We see that the statement and proof of Proposition~\ref{prop:monicconj}  apply to any $f(z) \in\ZZ[z]$ such that $f$ has content~1,  the leading coefficient is not a unit, and $f(0)=0$.
The proposition below provides, for each degree $d\geq 3$ and prime $p$ meeting the conditions of Theorem 1, 
a polynomial of degree $d$ with persistent bad reduction at $p$. All of our examples come from the family of  conservative dynamical Belyi polynomials.

\begin{proposition}\label{prop2} Let $p$ be prime and let $d=p^n \ell$, with $\ell>p$, $ p \nmid \ell$, and $n \geq 0$. Write $\ell = pq+r$ with $ 1 \leq r \leq p-1$ and $q \geq 1$. Let $k=p^n r$. Then the conservative dynamical Belyi polynomial $B_{d,k}$ has persistent bad reduction at $p$.
\end{proposition}

\begin{proof}
We have 
\[ 
B_{d,k}(z) = a_k z^d + a_{k-1} z^{d-1} + \dots + a_0 z^{d-k} \in \QQ[z] \subseteq \QQ_p[z].
\]
Following the proof of Proposition~\ref{prop:monicconj}, we choose
$\beta\in \overline{\QQ_p}$ such that $\beta^{d-1} = a_k$, and define $B_{d,k}^\phi(z)$ as in equation~\eqref{eq:BelyConj}.
Note that $a_0=\binom{d}{k}$, and consider the coefficient of $z^{d-k}$ for $B_{d,k}^\phi(z)$, namely $\frac{a_0}{\beta^{d-k-1}}$.  

As in Proposition~\ref{prop:monicconj}, let $v$ denote any valuation extending the $p$-adic valuation. We claim that $v(\beta^{d-k-1})>v(a_0)$. In other words, we claim the following:
\[ \frac{d-k-1}{d-1} v_p(a_k)>v_p(a_0),\]
or
\[ \frac{d-k-1}{d-1} v_p\left( \binom{d-1}{k} \right) > v_p \left( \binom{d}{k}\right).\]

Using Kummer's Theorem for binomial coefficients, we compute $v_p \left( \binom{d}{k}\right)$ by writing both $d$ and $k$ in base $p$ and counting the number of carries when we add $k$ to $d-k$. Since $k=p^n r$ and $d-k$ is a multiple of $p^{n+1}$, there are no carries when added together, so $v_p \left( \binom{d}{k}\right) = 0$. 

Looking at $a_k = \pm \binom{d-1}{k}$, we write $k$ and $d-k-1$ in base $p$, truncated modulo $p^{n+1}$:
\[ k=rp^n \]
\[ d-k-1 \equiv (p-1) p^n + (p-1) p^{n-1} + \dots + (p-1)  \pmod{p^{n+1}}.\]
Adding these together, we get at least one carry in the $p^n$ position,  so $v_p \left( \binom{d-1}{k}\right) \geq 1$.
Therefore $B_{d,k}^\phi(z)$ has bad reduction at $p$.  By Lemma~\ref{lem:pgr}, $B_{d,k}$ has persistent bad reduction at $p$.
\end{proof}

To prove Theorem 1, it remains to show that for any prime $p$, if $d= p^n \ell$ with $1\leq \ell<p$ and $n \geq 0$, then all PCF polynomials in $\overline{\QQ}[z]$ of degree $d$ have potential good reduction at $p$.

\begin{proof}[Proof of Theorem 1]

Let $p$ be prime and let $d,n,\ell \in \ZZ$ with $n \geq 0, 1\leq \ell < p$, and $d=p^n \ell \geq 2$. Let $f(z) \in \overline{\QQ}[z]$ be a PCF polynomial of degree $d$. Let $|\cdot|_p$ denote the absolute value on $\QQ$ normalized so that $|p|_p = \frac 1 p$.

After conjugation, we may assume $f$ is monic and $f(0)=0$. Write $f(z) = z^d + a_{d-1}z^{d-1} + \dots + a_1 z$. Since $f \in L[z]$ for some number field $L$, we may let  $|\cdot | $ denote the absolute value on $L$ extending $|\cdot|_p$, and let $v(\cdot)$ denote the associated valuation.

 Let $\{ \alpha_i\}$ denote the $d$ fixed points for $f$, listed with multiplicity. Let $\{ \gamma_i\}$ denote the $d-1$ critical points for $f$, also listed with multiplicity. Suppose for contradiction that $f$ is post-critically finite but has persistent bad reduction at $p$. Then we must have $|a_i|>1$ for some $i$. The Newton polygon for $f(z)-z$ must have a segment of positive slope, which implies that $f$ has at least one fixed point $\alpha_1$ satisfying $|\alpha_1|>1$. Looking at the rightmost segment of the Newton polygon for $f(z)-z$, suppose it connects the points $(k, v(a_k))$ and $(d,0)$. Let $r=-v(a_k)/(d-k)>0$. Then $f$ has exactly $d-k$ fixed points of absolute value $p^r$, and all other fixed points are smaller. Write $f(z) = z+\prod_{i=1}^d (z-\alpha_i)$. Then if $|x|>p^r$, we have $|f(x)| = |x + \prod_{i=1}^d (x-\alpha_i)| = |x|^d>|x|$, and so $x$ has unbounded orbit. Since $f$ is post-critically finite, and thus post-critically bounded with respect to this absolute value, all critical points for $f$ must satisfy $|\gamma_i| \leq p^r$.

Now, look at the rightmost segment of the Newton polygon for $f'(z)$. The slope of this segment is at least $\frac{v(d)-v(k a_k)}{d-k} \geq r$, with equality only possible if $n=v(d) = v(k)$ (in other words, if $k$ is a multiple of $p^n$). Since no critical point can be larger than the largest fixed point, we must have equality, and the rightmost segment of the Newton polygon for $f'(z)$ must go through the points $(d-1, n)$ and $(k-1, v(ka_k))$. (It could possibly extend further past this point.) So, we must have exactly $d-k$ fixed points of absolute value $p^r$, and at least $d-k$ critical points of absolute value $p^r$, counting with multiplicity, with all other fixed points and critical points of absolute value less than $p^r$.

Finally we look at the set of open disks $D(\alpha_i, p^r)$ of radius $p^r$ centered at each fixed point of~$f$ and recall that $0$ is a fixed point of $f$. Every critical point must lie in one of these disks because $f$ is post-critically finite, so it is also post-critically bounded. (If $x$ is a point outside this set of disks then the orbit of $x$ under $f$ is not bounded: $|x| \geq p^r$ and $|f(x)| = |x+ \prod_{i=1}^d (x- \alpha_i)| \geq p^{rd}$.) The Newton polygon argument above implies that $D(0, p^r)$ contains fewer critical points than fixed points (again, counting with multiplicity), since this disk has $k$ fixed points and at most $k-1$ critical points. Since every critical point must lie in one of these disks for $f$ to be post-critically finite, there must be a different disk $D(\alpha_2, p^r)$ containing at least as many critical points as fixed points. Conjugate by translation to move $\alpha_2$ to 0. The newly conjugated map has fewer critical points than fixed points  of absolute value $p^r$, which cannot occur since the newly conjugated map will satisfy the same criteria and thus the same argument holds. Thus, every coefficient $a_i$ of $f$ satisfies $|a_i| \leq 1$. Since  $f$ is monic, it has good reduction at~$p$. 
 \end{proof}

In the cases where $d<p$ or $d=p^n$, Theorem 1 recovers the previously known results from~\cite{anderson} and\cite{epstein}:

\begin{corollary} If $d=p^n$ for some prime $p$, or if $d<p$, then all PCF polynomials of degree $d$ have potential good reduction at $p$. 
\end{corollary}

\begin{example} Let $d = 18$. By Theorem 1, all PCF polynomials $f \in \overline{\QQ}[z]$ of degree $18$ have potential good reduction at primes larger than $18$, so we will consider primes $p\leq 17$. Note that Theorem 1 implies that all degree 18 PCF polynomials will have potential good reduction at~$p=3$, since $18 =3^2\cdot 2$, so in this case $\ell < p$. 
The following table lists the primes $3 \neq p \leq 17$, the values $n,l, r,$ and $k$ as described in Proposition~\ref{prop2}, and the Belyi polynomial $B_{d,k}$ with persistent bad reduction at $p$. 

\small
\begin{center}
\begin{table}[h]
\label{tab:table1}
\begin{tabular}{c|c|c|c|c|c|l}
$p$ & factor $d$ & $n$ & $\ell$ & $r \equiv \ell \mod p$ & $k = p^n r$ & $B_{18,k}$\\
\hline
2 & \tiny{$18=2^1\cdot 9$ }& 1 & 9 & 1 & 2 & \tiny{$136z^{18}-288z^{17}+153z^{16}$}\\
5 &\tiny{ $18=5^0\cdot 18$ } &0 & 18 & 3 & 3 & \tiny{$-680z^{18}+2160z^{17}-2295z^{16}+816z^{15}$}\\
7 & \tiny{$18=7^0\cdot 18$} &0 & 18 & 4 & 4 & \tiny{$2380z^{18} -10080z^{17}+16065z^{16}-11424z^{15}+3060z^{14}$}\\
11 &\tiny{ $18=11^0\cdot 18$} &0 & 18 & 7 & 7 & \tiny{$-1144z^{18}+144144z^{17}-459459z^{16}+816816z^{15} $}\\
& & & & & & \tiny{$\qquad -875160z^{14}+565488z^{13}-204204z^{12}+31824z^{11}$}\\
13 & \tiny{$18=13^0\cdot 18$} &0 & 18 & 5 & 5 & \tiny{$-6188z^{18}+32760z^{17}-68615z^{16}+74256z^{15}-39780z^{14}+8568z^{13}$}\\
17 &\tiny{ $18=17^0\cdot 18$} &0 & 18 & 1 & 1 & \tiny{$-17z^{18}+18z^{17}$}\\
\end{tabular}
\end{table}
\end{center}
\normalsize
\end{example}

Notice that the examples in Table ~\ref{tab:table1} may have persistent bad reduction at primes other than the listed prime. In fact, they will have persistent bad reduction at all primes that divide the leading coefficient, which can be seen in the proof of Proposition ~\ref{prop:monicconj}.

\section{Everywhere good reduction}
While Theorem 1 says that a PCF polynomial of degree $d$ \emph{may} have persistent bad reduction at $p$ for certain values of $p$, there is no $(d,p)$ pair for which a PCF polynomial of degree $d$ \emph{must} have persistent bad reduction at $p$. There are obvious unicritical PCF polynomials (such as $z^d$ or $z^d-1$ with $d$ even) that have good reduction everywhere, but other examples are more elusive.

\begin{question}[Adam Epstein, private communication]
Can we find, for every degree $d \geq 2$, a PCF polynomial $f(z) \in \overline\QQ[z]$ of degree $d$ that has potential good reduction at every prime, where $f(z)$ is not (conjugate to) a unicritical polynomial or a composition of unicritical polynomials?
\end{question}

Using compositions of unicritical polynomials, we can construct infinitely many PCF functions with everywhere good reduction by the following method~\cite{Pingram}: 
Consider the composition of $z^d + a$ and $z^e + b$ :
$f(z) = (z^d+a)^e + b$. If $a =0$ then $f$ is unicritical, so assume $a\neq 0$. The critical points are at the $d^\text{th}$ roots of $-a$ and at $0$.  If $f$ is PCF, then $a$ and $b$ are in fact algebraic integers, so $f$ will have good reduction at every prime. By looking at the algebraic conditions given by forcing finite critical orbits of varying lengths, we should have infinitely many pairs $(a, b) \in \overline\QQ^2$ such that $f$ is PCF. And of course, we can compose more than two unicritical functions to get even more examples.

We are not able to fully answer Epstein's question, but we provide in this section some results in that direction. The proposition below provides, for each pair $(d,p)$ with $d\geq 3$ and $p$ prime,  an example of a bicritical PCF polynomial of degree $d$ with good reduction at $p$. In infinitely many cases (when $d=p^k + 1$ for some prime $p$ and positive integer $k$), this example has good reduction everywhere, but in general this proposition does not provide an example of a PCF polynomial in every degree with everywhere good reduction.

We do know that our bicritical examples in the following proposition are not compositions of unicritical polynomials because such compositions are not bicritical. As noted above, if we compose two unicritical polynomials of the form $z^d+a$ and $z^e+b$, where $a \neq 0$, the critical points of the composition $(z^d+a)^e+b$ are at the $d^\text{th}$ roots of $-a$ and at $0$, so we have more than two distinct critical points whenever $d>1$.
\begin{proposition} \label{prop:goodred}
For any integer $d \geq 3$ and prime $p$, there exists a bicritical post-critically finite polynomial $f \in \overline{\mathbb{Q}}[z]$ of degree $d$ with good reduction at $p$.
\end{proposition}

\begin{proof}
Let $p$ be prime and consider the following polynomial of degree $d$:
\[ f(z) = a(-(d-1)z^d+dz^{d-1})+\frac{d}{d-1}.\]

This is modified from the conservative Belyi polynomial $B_{d,1}$. Like $B_{d,1}$, it has just two finite critical points at $0$ and $1$. The first critical point, 0, is preperiodic, as $f(0) = \frac{d}{d-1}$, which is a fixed point. The second critical point $1$ maps to $a + \frac{d}{d-1}$. We may choose $a$ such that $a+\frac{d}{d-1}$ is also a fixed point for $f$, thus making $f$ PCF. We claim that there is one such choice of $a$ such that the resulting PCF polynomial has potential good reduction at $p$.

The $a$ values that result in the critical orbit described above are the roots of the following polynomial in $a$:
\[ (d-1)a\left(a+\frac{d}{d-1}\right)^{d-1}+1.\]
Looking at the Newton polygon for the polynomial above at $p$, we see that if $p \nmid (d-1)$, then $a$ is a $p$-adic unit and $f$ already has good reduction at $p$. If instead $v_p(d-1) = m>0$, then the Newton polygon has a line segment connecting $(0,0)$ to its lowest point at $(1, -(d-2)m)$. This implies that there exists one root $a$ of this polynomial such that $v_p(a) = (d-2)m$, and all other roots have negative $p$-adic valuation.

If we choose the $a$ value described above such that $v_p(a) = (d-2)m$, then $f$ will have potential good reduction at $p$. We may see this by conjugating $f$ by a scaling map to make it monic. Let $\beta$ be such that $\beta^{d-1} = -a(d-1)$, the leading coefficient of $f$. Then let $\phi(z) = \frac{z}{\beta}$ and define the conjugate map $f^\phi(z) = \phi^{-1} \circ f \circ \phi (z)$. We then have
\[ f^\phi(z) = z^d - \frac{d}{d-1} \beta z^{d-1} + \frac{d}{d-1}\beta.\]
Computing the $p$-adic valuation of $\frac{d}{d-1}\beta$, we see that the coefficients of $f^\phi$ are all $p$-adic units, and thus $f^\phi$ is a PCF polynomial of degree $d$ with good reduction at $p$.
\end{proof}

In particular, when $d-1$ is a prime power, we have the following corollary:

\begin{corollary}
If $d=p^k+1$ for some prime $p$ and some integer $k \geq 1$, then there exists a bicritical PCF polynomial of degree $d$ with potential good reduction everywhere.
\end{corollary}

\begin{proof}
Let $d=p^k+1$ and define $f(z)$ as in the proof of Proposition~\ref{prop:goodred}, with the coefficient $a$ chosen so that $f$ has potential good reduction at $p$. Note that as defined, $f$ has good reduction at $q$ for any prime $q$ not dividing $d-1$. Since $d-1=p^k$ in this case, $f$ has good reduction at all primes $q \neq p$, and thus has potential good reduction everywhere.
\end{proof}

In the degree 3 case, the example given in Proposition~\ref{prop:goodred} is conjugate to the Chebyshev polynomial $T_3(x) = 4x^3-3x$. Since Chebyshev polynomials could be a special case that we'd like to exclude just as we excluded compositions of unicritical polynomials, we note that none of our higher-degree examples are Chebyshev polynomials. We can see this because our examples are all bicritical, while the degree-$d$ Chebyshev polynomial has $d-1$ distinct critical points: Since the Chebyshev polynomial $T_d(x)$ satisfies $T_d(\cos \theta) = \cos(d\theta)$, we see that $T_d$ has $d$ roots and $d-1$ distinct critical points between $-1$ and $1$.

\section{Heights of Conservative Polynomials}\label{sec:heights}

Conservative dynamical Belyi polynomials can also be used to answer two questions posed by Silverman in~\cite{MdADS} about the heights of conservative polynomials. Recall that a conservative polynomial $f$ is a polynomial for which every critical point $\gamma_i$ satisfies $f(\gamma_i) = \gamma_i$.

\begin{definition}\label{def:height}
The (absolute projective) height of a degree-$d$ polynomial
\[ f(z) = c \prod_{i=1}^d (z-\alpha_i) \in \overline{\QQ}[z] \]
is given by
\[ h(f) = \frac1d \sum_{i=1}^d h(\alpha_i).\]
Note that if $f \in \QQ[z]$ is irreducible, then $h(f) = h(\alpha)$ for any root $\alpha$ of $f$.
\end{definition}

To compute the height of a polynomial, we can use the Mahler measure. (For a reference on heights and Mahler measure, see~\cite{heightreference}, particularly Proposition 1.6.6.)

\begin{definition}[Mahler measure]
If $f(z) = c \displaystyle\prod_{i=1}^d (z-\alpha_i)$, then 
\[M(f)=|c| \displaystyle\prod_{|\alpha_i| \geq 1} |\alpha_i|. \] 

We have the following relationship between height and Mahler measure for polynomials with integral coefficients of content 1:
 \begin{equation}\label{eq:mahler} h(f)=\frac{1}{d} \log M(f).\end{equation}
\end{definition}

We consider the following questions of Silverman:

\begin{question} [\protect{Question 6.55 in~\cite{MdADS}}]
\label{q1}
Define a polynomial $f$ to be normalized if it is monic and $f(0)=0$. Are the following statements   about the set $\mathcal{C}_d^{\text{poly}}$ of normalized conservative polynomials of degree $d$ in $\overline{\mathbb{Q}}[z]$ true?

\begin{equation}\label{eq:question1}
\lim_{d \to \infty} \max_{f \in \mathcal{C}_d^{\text{poly}}} h(f) = 0?
\end{equation}

\begin{equation}\label{eq:limitquestion}
\limsup_{d \to \infty} \max_{f \in \mathcal{C}_d^{\text{poly}}} \frac{h(f)}{(\log{d})/d} < \infty?
\end{equation}

\end{question}

Note that Silverman's definition of ``normalized polynomial'' differs from the definition used in the previous section. Our  dynamical Belyi maps $B_{d,k}$ are conservative and satisfy $B_{d,k}(0) = 0$, but they are not monic. To put them in the desired normal form, we will use scaled conjugates of these maps to provide negative answers to both parts of Question~\ref{q1}. The leading coefficient of $B_{d,k}$ is $a_k = (-1)^k \binom{d-1}{k}$, so as before we choose $\beta$ such that $\beta^{d-1} = a_k$ and conjugate $B_{d,k}$ by the scaling map $\phi(z) = \frac{z}{\beta} $ to obtain the polynomial $B_{d,k}^\phi$, just as in equation~\eqref{eq:BelyConj}. The polynomial $B_{d,k}^\phi$ is a monic conservative polynomial satisfying the hypotheses in Question~\ref{q1}.

We begin by bounding the difference between $h(B_{d,k})$ and $h(B^\phi_{d,k})$.

\begin{lemma}\label{lem:heightdiff} Let $B^\phi_{d,k}$ be defined as in equation~\eqref{eq:BelyConj}. Then
\[ \left| h(B_{d,k}) - h(B_{d,k}^\phi) \right| \leq \frac{k}{d(d-1)} \log{\left(\binom{d-1}{k}\right)}.\]
\end{lemma}

\begin{proof}
If $\{ \alpha_1, \alpha_2, \dots, \alpha_k\}$ are the nonzero roots of $B_{d,k}$, then $\{\alpha_1\beta, \alpha_2 \beta, \dots, \alpha_k \beta\}$ are the non-zero roots of $B^\phi_{d,k}$, and using the triangle inequality for heights we obtain
\begin{equation} h(B^\phi_{d,k}) = \frac1d \sum_{i=1}^k h(\alpha_i \beta) \geq \frac1d \sum_{i=1}^k (h(\alpha_i) - h(\beta)) = h(B_{d,k}) - \frac{k}{d}h(\beta),
\label{eq:triangleineq1}
\end{equation}
and similarly, 
\begin{equation} h(B^\phi_{d,k}) = \frac1d \sum_{i=1}^k h(\alpha_i \beta) \leq \frac1d \sum_{i=1}^k (h(\alpha_i) + h(\beta)) = h(B_{d,k}) + \frac{k}{d}h(\beta).
\label{eq:triangleineq2}
\end{equation}
Recall that $\beta^{d-1} = (-1)^k \binom{d-1}{k}$, and so $h(\beta) = \frac{1}{d-1} \log{\left(\binom{d-1}{k}\right)}$. Substituting this into the inequalities above gives the desired result.
\end{proof}

\begin{thm} \label{thm:bheight} Let $B^\phi_{d,k}$ be defined as in equation~\eqref{eq:BelyConj}. Then
\[ h(B_{d,k}^\phi) > \frac{1}{d} \log{\left( {d \choose k}\right)} \left(1-\frac{k}{d-1}\right).\]
\end{thm}

\begin{proof}
First we compute a lower bound for $h(B_{d,k})$ using the Mahler measure. Define the polynomial $F_{d,k}(z) = B_{d,k}(z)/z^{d-k}$. Since $B_{d,k}$ and $F_{d,k}$ have the same nonzero roots, we see from Definition~\ref{def:height} that 
\[ h(B_{d,k}) = \frac{k}{d} h(F_{d,k}). \]

We have normalized such that $F_{d,k}(1) = B_{d,k}(1) = 1$, so $F_{d,k}(z) \in \mathbb{Z}[z]$ has content 1. Therefore, we can use the Mahler measure to compute the height of $F_{d,k}$. Using equation~\eqref{eq:mahler}, we see that
\begin{equation}
 h(B_{d,k}) = \frac{k}{d}\cdot  \frac{1}{k} \log{M(F_{d,k})} = \frac1d \log{M(F_{d,k})}.
 \label{eqn:fb-mahler}
 \end{equation}

 Let $\alpha_1, \alpha_2, \ldots ,\alpha_k$ denote the roots of $F_{d,k}$, listed with multiplicity. Note that $F$ contains at least one root outside the unit disk because the absolute value of the product of the roots is $|a_0/a_k| = \frac{d}{d-k} >1$.  
  If $\abs{\alpha_i} \geq 1$ for all $1 \leq i \leq k$ then $M(F_{d,k})$ is simply equal to the constant term of the polynomial $F_{d,k}$, namely $\binom{d}{k}$. Otherwise, this provides a lower bound for $M(F_{d,k})$, and so we have the following:

\[ h(B_{d,k})=  \frac1d \log{M(F_{d,k})} = \frac1d \log{\left(\frac{\binom{d}{k}}{\prod_{|\alpha_i|<1} |\alpha_i|} \right)}  \geq \frac1d  \log{\left( {d \choose k}\right)} .\]

Next, we use Lemma~\ref{lem:heightdiff} to obtain a lower bound for the height of the monic conjugate $B^\phi_{d,k}$:

\[  h(B^\phi_{d,k}) \geq h(B_{d,k})- \frac{k}{d(d-1)} \log{\left(\binom{d-1}{k}\right)} \geq \frac1d  \log{\left( {d \choose k}\right)} - \frac{k}{d(d-1)}  \log{\left(\binom{d-1}{k}\right)} \]
\[> \frac{1}{d} \log{\left( {d \choose k}\right)} \left(1-\frac{k}{d-1}\right).\qedhere \]

\end{proof}

Taking $d=2k$ to maximize the value of the binomial coefficient, we get the following result, which provides a negative answer to the question posed in equation~\eqref{eq:question1}.

\begin{cor}\label{thm:height1} Let $d=2k$, where $k$ is a positive integer. Then,
\[ \liminf_{k \to \infty} h(B_{2k,k}^\phi) \geq \frac{1}{2} \log{2}.\]
\end{cor}

\begin{proof}
The inequality below follows from Theorem~\ref{thm:bheight}:
\[ h(B_{2k,k}^\phi) > \frac{1}{2k} \log{\left( {2k \choose k}\right)} \left(1-\frac{k}{2k-1}\right) = \frac{k-1}{2k(2k-1)} \log{\binom{2k}{k}} .\]

We use Stirling's approximation, which gives $\binom{2k}{k} \sim \frac{2^{2k}}{\sqrt{\pi k}}$, to compute the limit.
\[ \lim_{k \to \infty} \left( \frac{k-1}{2k(2k-1)} \log{\binom{2k}{k}} \right) =\lim_{k \to \infty}  \left(\frac{k-1}{2k(2k-1)} \log{\left(\frac{2^{2k}}{\sqrt{\pi k}}\right)}  \right)\]
\[ =\lim_{k \to \infty} \left(\frac{k-1}{2k-1} \log{2} -\frac{k-1}{4k(2k-1)}\log{(\pi k)} \right)=\frac{1}{2} \log{2} . \qedhere\]

\end{proof}

Corollary~\ref{thm:height1} provides a negative answer to Silverman's first question. This implies a negative answer to the second question as well, though it is interesting to investigate how the ratio of $h(B_{d,k}^\phi)$ to $\left(\log{d}\right)/d$ varies for different choices of $k$. Being a bit more precise in estimating $h(B_{d,k})$ than in the proof of Theorem~\ref{thm:bheight} , we get the result below.

\begin{thm} \label{thm:height} Fix an integer $k \geq 1$ and let $d \geq k+2$. Let $B^\phi_{d,k}$ be defined as in equation~\eqref{eq:BelyConj}. Then

 \[ \lim_{d \to \infty} \frac{h(B_{d,k}^\phi)}{\frac 1 d \log{d}} =k.\]

\end{thm}

To prove this, we start by showing that this limit holds for $B_{d,k}$ in place of $B_{d,k}^\phi$. Then we will obtain the result for $B_{d,k}^\phi$ using Lemma~\ref{lem:heightdiff}.

\begin{thm}\label{thm:height2}
Fix an integer $k \geq 1$ and let $d \geq k+2$. Then,
 \[ \lim_{d \to \infty} \frac{h(B_{d,k})}{\frac 1 d \log{\left( {d \choose k}\right)}} =1.\]
 \end{thm}

\begin{proof}

Define  $F_{d,k}(z) = B_{d,k}(z)/z^{d-k}$ as in the proof of Theorem~\ref{thm:bheight}. There, we saw that
\[ h(B_{d,k})=  \frac1d \log{M(F_{d,k})} = \frac1d \log{\left(\frac{\binom{d}{k}}{\prod_{|\alpha_i|<1} |\alpha_i|} \right)} , \]
where $\alpha_i$ are the roots of $F_{d,k}$.
 

Since
\[ h(F_{d,k}) = \frac{1}{k} \log{\left(\frac{\binom{d}{k}}{\prod_{|\alpha_i|<1} |\alpha_i|} \right)} ,\]
it follows that
\begin{equation}
\frac{h(B_{d,k})}{\frac 1 d \log{\left( {d \choose k}\right)}} = \frac{k \cdot h(F_{d,k})}{\log{\left( {d \choose k}\right)}} =  \frac{\log{\left(\frac{\binom{d}{k}}{\prod_{|\alpha_i|<1} |\alpha_i|} \right)}}{\log{\left( {d \choose k}\right)}} = 1-\frac{\log{\left(\prod_{|\alpha_i|<1} |\alpha_i| \right)}}{\log{\left( {d \choose k}\right)}}.
\label{eqn:bf-height}
\end{equation}

 Using a Cauchy bound, if $\alpha$ is a root of the reciprocal polynomial $z^kF_{d,k}\left(\frac{1}{z}\right)$, then
\[
\abs{\alpha}\leq 1+\max\limits_{1\leq i \leq k} \left\lbrace \abs{\frac{a_i}{a_0}} \right\rbrace \leq 2^k+1<2^{k+1}.
\]
Therefore,  the roots of $F_{d,k}$ are bounded away from zero; specifically, $\abs{\alpha_i} \geq 2^{-(k+1)}$. 
We calculate:
\begin{equation*} 
0 \geq \lim_{d \to \infty} \frac{\log\prod\limits_{\abs{\alpha_i}<1} \abs{\alpha_i}}{\log\left({d\choose k}\right)} \geq \lim\limits_{d\rightarrow \infty} \frac{\log 2^{-k^2-k}}{\log \left( {d\choose k}\right)} = 0.
\end{equation*}

Combining this with equation~\eqref{eqn:bf-height} gives the desired result.
\end{proof}

\begin{cor}\label{cor:heightlimit} Fix an integer $k \geq 1$. Then the following limit holds:
\[ \lim_{d \to \infty} \frac{h(B_{d,k})}{(\log{d})/d} = k.\]
\end{cor}

\begin{proof} From Theorem~\ref{thm:height2}, we have
 \[ \lim_{d \to \infty} \frac{h(B_{d,k})}{\frac 1 d \log{\left( {d \choose k}\right)}} =1.\]
 So, 
  \[ 
  \lim_{d \to \infty} \frac{h(B_{d,k})}{(\log{d})/d} = \lim_{d\to \infty} \frac{h(B_{d,k})}{\frac 1 d \log{\left( {d \choose k}\right)}}  \cdot \frac{\frac 1 d \log{\left( {d \choose k}\right)}}{(\log{d})/d} = \lim_{d \to \infty} \frac{\log{\left( {d \choose k}\right)}}{\log{d}} = k.\qedhere
  \]
\end{proof}

We are now able to prove Theorem~\ref{thm:height} by combining the results in Corollary~\ref{cor:heightlimit} and Lemma~\ref{lem:heightdiff}.

\begin{proof}[Proof of Theorem \ref{thm:height}]

Using Lemma~\ref{lem:heightdiff} we get
\[ \lim_{d \to \infty} \frac{h(B_{d,k})-\frac{k \log{\binom{d-1}{k}}}{d(d-1)}}{\log d/d}  \leq \lim_{d\to\infty} \frac{h(B_{d,k}^\phi)}{\log d /d} \leq \lim_{d \to \infty} \frac{h(B_{d,k})+\frac{k \log{\binom{d-1}{k}}}{d(d-1)}}{\log d/d}. \]

Noting that $\lim_{d \to \infty} \frac{\frac{k \log{\binom{d-1}{k}}}{d(d-1)}}{\log{d}/d} = 0$, we simplify and use Corollary~\ref{cor:heightlimit} to see that we have

\[ \lim_{d \to \infty} \frac{h(B^\phi_{d,k})}{\log d/d}  = \lim_{d\to\infty} \frac{h(B_{d,k})}{\log d /d} = k.\qedhere \]

\end{proof}

\subsection*{Acknowledgements}
The project was begun during a SQuaRE at the American Institute of Mathematics. The authors thank AIM for providing a supportive environment. The authors also thank the anonymous referees for their careful reading and helpful comments which improved the paper enormously.

This material is based upon work supported by and while the second author served at the National Science Foundation. Any
opinion, findings, and conclusions or recommendations expressed in this material are those of the authors
and do not necessarily reflect the views of the National Science Foundation.
The second author's work partially supported by Simons Foundation collaboration grant \#359721. The third author is partially supported by an AMS-Simons travel grant.

\subsection*{Statements and Declarations}
The authors have no conflicts of interest to declare that are relevant to the content of this article.

\bibliographystyle{plain}
\bibliography{belyipolys}

\end{document}